\documentclass[11pt]{article}
\usepackage{amssymb,amsmath,amsthm,latexsym}
\newtheorem{theorem}{Theorem}[section]
\newtheorem{lemma}[theorem]{Lemma}

\newtheorem{remark}[theorem]{Remark}

\newtheorem{prop}[theorem]{Proposition}

\begin{document}

\def\cA{{\mathcal A}}
\def\E{\mathbf E}
\def\F{\mathbf F}
\def\P{\mathbf P}
\def\N{\mathbf N}
\def\Z{\mathbf Z}
\def\K{K}
\def\L{\mathbf L}
\def\Q{\mathbf Q}
\def\U{\mathbf U}
\def\V{\mathbf V}
\def\W{\mathbf W}
\def\X{\mathbf X}
\def\Y{\mathbf Y}
\def\cB{\mathcal B}
\def\cC{\mathcal C}
\def\cD{\mathcal D}
\def\cE{\mathcal E}
\def\cF{{\mathcal F}}
\def\cG{{\mathcal G}}
\def\cH{\mathcal H}
\def\cL{\mathcal L}
\def\cM{{\mathcal M}}
\def\cK{\mathcal K}
\def\cO{\mathcal O}
\def\cP{\mathcal P}
\def\cX{\mathcal X}
\def\cY{\mathcal Y}
\def\cU{\mathcal U}
\def\cV{\mathcal V}
\def\cT{\mathcal T}
\def\cR{\mathcal R}
\def\cS{{\mathcal S}}
\def\cW{\mathcal W}
\def\cFR{\mathcal{Fr}}
\def\PG{{\rm PG}}
\def\PGL{{\rm PGL}}
\def\GF{{\rm GF}}
\def\AG{{\rm AG}}
\def\GL{{\rm GL}}
\def\char{{\rm char}}
\def\deg{{\rm deg}}
\def\det{{\rm det}}
\def\diag{{\rm diag}}
\def\dim{{\rm dim}}
\def\div{{\rm div}}
\def\fq{\mathbf F_{q}}
\def\fqi{\mathbf F_{q^i}}
\def\fqd{\mathbf F_{q^d}}
\def\frob{{\em Fr}}
\def\ord{{\rm ord}}
\def\rank{{\rm rank}}
\def\res{{\rm res}}
\def\sec{\text{\it Sec\, }}
\def\sq{\sqrt{q}}
\def\supp{{\rm Supp}}
\def\Prf{{Proof. \quad}}
\def\hxi{{\xi}^{\cH}}
\def\hgamma{{\gamma}^{cH}}
\def\hS{{S}^{cH}}
\def\PO{{\rm P\Omega^-}}
\title{New Cameron--Liebler line classes with parameter $\frac{q^2+1}{2}$}
\author{A. Cossidente \and F. Pavese}
\date{}
\maketitle

\begin{abstract}
New families of Cameron--Liebler line classes of $\PG(3,q)$, $q\ge 7$ odd, with parameter $(q^2+1)/2$
are constructed.
\end{abstract}



\section{Introduction} \label{sect:intro}

In this paper we deal with Cameron--Liebler line classes of $\PG(3,q)$. The notion of Cameron--Liebler line class was introduced in the seminal paper \cite{CL} with the purpose of classifying those collineation groups of $\PG(3,q)$ having the same number of orbits on points and lines. On the other hand, a classification of Cameron--Liebler line classes would yield a classification of symmetric tactical decompositions of points and lines of $\PG(n,q)$ and that of certain families of weighted point sets of $\PG(3,q)$ \cite{P}. Cameron--Liebler line classes are also related to other combinatorial structures such as two--intersection sets, strongly regular graphs and two--weight codes \cite{BKLP}, \cite{MR}.

In $\PG(3,q)$, a {\em Cameron--Liebler line class} $\cal L$  with parameter $x$ is a set of lines such that every spread of $\PG(3,q)$ contains exactly $x$ lines of $\cal L$, \cite{CL}, \cite{P}. There exist trivial examples of Cameron--Liebler line classes $\cal L$ with parameters $x=1,2$ and $x=q^2,q^2-1$. A Cameron--Liebler line class with parameter $x=1$ is either the set of lines through a point or the set of lines in a plane. A Cameron--Liebler line class with parameter $x=2$ is the union of the two previous example, if the point is not in the plane \cite{CL}, \cite{P}. In general, the complement of a Cameron--Liebler line class with parameter $x$ is a Cameron--Liebler line class with parameter $q^2+1-x$ and the union of two disjoint Cameron--Liebler line classes with parameters $x$ and $y$, respectively, is a Cameron--Liebler line class with parameter $x+y$. 

It was conjectured that no other examples of Cameron--Liebler line classes exist \cite{CL}, but Bruen and Drudge \cite{BD} found an infinite family of Cameron--Liebler line classes with parameter $x=(q^2+1)/2$, $q$ odd. Bruen--Drudge's example admits the group $G=\PO(4,q)$, stabilizing an elliptic quadric ${\cal Q}^-(3,q)$ of $\PG(3,q)$, as an automorphism group. 

In \cite{GP}, Govaerts and Penttila found a Cameron--Liebler line class with parameter $x=7$ in $\PG(3,4)$. 

Infinite families of Cameron--Liebler line classes with parameter $(q^2-1)/2$ were found for $q \equiv 5$ or $9\pmod{12}$ in \cite{DDMR}, \cite{FMX}. By construction, such a family $\cal X$ never contains the lines $\cal Y$ in a plane $\Pi$ and the lines $\cal Z$ through a point $z \not\in \Pi$. Therefore, ${\cal X}\cup{\cal Y}$ and ${\cal X}\cup{\cal Z}$  are both examples of  Cameron--Liebler line classes with parameter $(q^2+1)/2$.

For non--existence results of Cameron--Liebler line classes see \cite{GM}, \cite{M} and references therein.

Recently, by perturbating the Bruen--Drudge's example, a new infinite family of Cameron--Liebler line classes with parameter $(q^2+1)/2$, $q\ge 5$ odd, has been constructed independently in \cite{CP} and \cite{GMP}. Such a family admits the stabilizer of a point of ${\cal Q}^-(3,q)$ in the group $G$, say $K$.

Here, we introduce a new derivation technique for Cameron--Liebler line classes with parameter $(q^2+1)/2$, see Theorem \ref{cl}. Applying such a derivation to the Bruen--Drudge's example, we construct a new family of Cameron--Liebler line classes with parameter $(q^2+1)/2$, $q\ge 7$ odd, not equivalent to the examples known so far and admitting a subgroup of $K$ of order $q^2(q+1)$.

Under the Klein correspondence between the lines of $\PG(3,q)$  and points of a Klein quadric ${\cal Q}^+(5,q)$,  a Cameron--Liebler line class with parameter $i$ gives rise to a so--called {\em $i$--tight set} of ${\cal Q}^+(5,q)$.

A set $\cT$ of points of ${\cal Q}^+(5,q)$ is said to be {\em $i$--tight} if
$$
\vert P^\perp\cap{\cal T}\vert=
\begin{cases} i(q+1)+q^2 & if \quad P\in {\cal T}\\
i(q + 1) & if \quad P\not\in {\cal T}\end{cases},
$$
where $\perp$ denotes the polarity of $\PG(5,q)$ associated with ${\cal Q}^+(5,q)$. 
	
For more results on tight sets of polar spaces, see \cite{BKLP}.

Throughout the paper $q$ is a power of an odd prime.

	\section{The Bruen--Drudge's example}

Let $X_1,\dots, X_4$ be homogeneous projective coordinates in $\PG(3,q)$. Let $\omega$ be a non--square element of $\GF(q)$. Let $\cE$ be the elliptic quadric of $\PG(3,q)$ with equation $X_1^2-\omega X_2^2+X_3X_4=0$ and quadratic form $Q$.  Each point $P\in\cE$ lies on $q^2$ secant lines to $\cE$, and so lies on $q+1$ tangent
lines. Let ${\cal L}_P$ be a set of $(q+1)/2$ tangent lines to $\cE$ through $P$, and let $E$ be 
the set of external lines to $\cE$; then the set
$$
\bigcup_{P\in\cE}{\cal L}_P \cup E
$$
has size $(q^2+1)(q^2+q+1)/2$ which is the number of lines of $\PG(3,q)$ in a Cameron--Liebler line class with parameter $(q^2+1)/2$. By suitably selecting the sets ${\cal L}_P$, it is in fact a Cameron--Liebler line class \cite{BD}. Let $G=\PO(4,q)$ be the commutator subgroup of the full stabilizer of $\cE$ in $\PGL(4,q)$. The group $G$ has three orbits on points of $\PG(3,q)$, i.e., the points of $\cE$ and other two orbits $O_s$ and $O_n$, both of size $q^2(q^2+1)/2$. The two orbits $O_s$, $O_n$ correspond to points of $\PG(3,q)$ such that the evaluation of the quadratic form $Q$ is a non--zero square or a non--square in $\GF(q)$, respectively. 
We say that a point $X$ of $\PG(3,q)$ is a {\em square point} or a {\em non--square point} with respect to a given quadratic form $F$ according as the evaluation $F(X)$ is a non--zero square or a non--square of $\GF(q)$. 
In its action on lines of $\PG(3,q)$, the group $G$ has four orbits: two orbits, say ${\cal L}_0$ and ${\cal L}_1$, both of size $(q+1)(q^2+1)/2$, consisting of lines tangent to $\cE$ and two orbits, say ${\cal L}_2$ and ${\cal L}_3$, both of size $q^2(q^2+1)/2$ consisting of lines secant or external to $\cE$, respectively.
The block-tactical decomposition matrix for this orbit decomposition
is $$\left[\begin{array}{ r r r r} 1&1&2&0\\
q&0&\frac{q-1}{2}&\frac{q+1}{2}\\
0&q&\frac{q-1}{2}&\frac{q+1}{2}\\
\end{array}\right],$$

and hence the point-tactical decomposition matrix is

$$\left[\begin{array}{ r r r r} \frac{q+1}{2}&\frac{q+1}{2}&q^2&0\\
q+1&0&\frac{q(q-1)}{2}&\frac{q(q+1)}{2}\\
0&q+1&\frac{q(q-1)}{2}&\frac{q(q+1)}{2}\\
\end{array}\right].$$

Simple group--theoretic arguments show that a line of ${\cal L}_0$ (${\cal L}_1$) contains $q$ points of $O_s$ ($O_n$), that a line secant to $\cE$ always contains $(q-1)/2$ points of $O_s$ and $(q-1)/2$ points of $O_n$ and that a line external to $\cE$ contains  $(q+1)/2$ points of $O_s$ and $(q+1)/2$ points of $O_n$. From the orbit--decompositions above it is an easy matter to prove the following results:
\begin{lemma}\label{lemma1}
Let $\ell$ be a line that is either secant or external to $\cE$. Then the number of lines of $\cL_0$ (or of $\cL_1$) meeting $\ell$ equals $(q+1)^2/2$. 
\end{lemma}
\begin{lemma}\label{lemma2}
Let $\ell$ be a line of $\cL_0$ (resp. $\cL_1$). Then 
\begin{itemize}
\item the number of lines of $\cL_0$ (resp. $\cL_1$) meeting $\ell$ in a point equals $q^2+(q-1)/2$; 
\item the number of lines of $\cL_1$ (resp. $\cL_0$) meeting $\ell$ in a point equals $(q+1)/2$; 
\end{itemize}
\end{lemma}

Let ${\cal L'}={\cal L}_0\cup{\cal L}_3$. Using Lemma \ref{lemma1} and Lemma \ref{lemma2}, it can be seen that ${\cal L'}$ is the Cameron--Liebler line class constructed in \cite{BD}. 

In particular, ${\cal L'}$ has the following three characters with respect to line--sets in planes of $\PG(3,q)$: 
$$
q^2+\frac{q+1}{2}, \frac{q(q-1)}{2}, \frac{q(q+1)}{2}+1, 
$$
and 
$$
\frac{q+1}{2}, \frac{q(q+1)}{2}, \frac{q(q+1)}{2}+q+1,
$$
with respect to line--stars of $\PG(3,q)$.

Consider a point of $\cE$ (since $G$ is transitive on $\cE$ we can choose it as the point $U_3=(0,0,1,0)$) and let $\pi$ be the plane with equation $X_4=0$. Then $\pi$ is tangent to $\cE$ at the point $U_3$. The plane $\pi$ contains a set ${\cal L}_3'$ consisting of $q^2$ lines of ${\cal L}_3$. Let ${\cal L}_2'$ be the set of $q^2$ lines of ${\cal L}_2$ through $U_3$. In \cite{CP}, we showed that ${\cal L}''=({\cal L'}\setminus {\cal L}_3') \cup {\cal L}_2'$ is a Cameron--Liebler line class of with parameter $(q^2+1)/2$, admitting the group $K$ as an automorphism group.

In particular, ${\cal L}''$ has the following five characters with respect to line--sets in planes of $\PG(3,q)$:
$$
\frac{q+1}{2}, \frac{q(q-1)}{2}-1, \frac{q(q+1)}{2}, \frac{q(q+1)}{2}+q+1, q^2+\frac{q-1}{2},
$$
and 
$$
\frac{q+3}{2}, \frac{q(q-1)}{2}, \frac{q(q+1)}{2}+1, \frac{q(q+1)}{2}+q+2, q^2+\frac{q+1}{2},
$$ 
with respect to line--stars of $\PG(3,q)$. It turns out that, if $q >3$, these characters are distinct from those of a Bruen--Drudge Cameron--Liebler line class.

\section{The new family}

In this section we introduce a new derivation technique which will allow us to construct a new infinite family of Cameron--Liebler line classes with parameter $(q^2+1)/2$, $q \ge 7$. 

With the notation introduced in the previous section, let $\cE_{\lambda}$ be the elliptic quadric with equation $X_1^2-\omega X_2^2+\lambda X_4^2+X_3X_4=0$, $\lambda\in\GF(q)$. Then the non--degenerate quadrics $\cE=\cE_0,\cE_{\lambda}$, $\lambda\in\GF(q)\setminus\{0\}$, together with $\pi$, form a pencil $\cal P$ of $\PG(3,q)$. The base locus of $\cal P$ is the point $U_3$.
Let $\perp$ (resp. $\perp_{\lambda}$) be the orthogonal polarity associated to $\cE$ (resp. $\cE_{\lambda}$).
Note that $U_3^{\perp_{\lambda}}=\pi$, $\forall \lambda\in\GF(q)\setminus\{0\}$. 

With a slight abuse of notation we will denote in the same way a plane and the set of lines contained in it. 
\begin{lemma}\label{square1}
If $P \in O_s$, then  
$$
\vert P^\perp \cap \cL_0 \vert=
\begin{cases} 
q+1 & \mbox{ if } \quad q \equiv -1 \pmod 4 \\
0 & \mbox{ if } \quad q \equiv 1 \pmod 4
\end{cases} , 
$$   
$$
\vert P^\perp \cap \cL_1 \vert=
\begin{cases} 
0 & \mbox{ if } \quad q \equiv -1 \pmod 4 \\
q+1 & \mbox{ if } \quad q \equiv 1 \pmod 4
\end{cases} .
$$   
If $P \in O_n$, then  
$$
\vert P^\perp \cap \cL_0 \vert=
\begin{cases} 
0 & \mbox{ if } \quad q \equiv -1 \pmod 4 \\
q+1 & \mbox{ if } \quad q \equiv 1 \pmod 4
\end{cases} , 
$$   
$$
\vert P^\perp \cap \cL_1 \vert=
\begin{cases} 
q+1 & \mbox{ if } \quad q \equiv -1 \pmod 4 \\
0 & \mbox{ if } \quad q \equiv 1 \pmod 4
\end{cases} .
$$   
\end{lemma}
\begin{proof}
It is enough to show that if $t \in \cL_0$, then $t^\perp$ belongs either to $\cL_0$ or to $\cL_1$, according as $q \equiv -1 \pmod 4$ or $q \equiv 1 \pmod 4$. Since $G$ is transitive on elements of $\cL_0$, let $t$ be the line joining $U_3$ with $U_1 = (1,0,0,0)$. Then $t^\perp$ is the line joining $U_3$ with $U_2 = (0,1,0,0)$, which belongs to $\cL_0$ if and only if $-1$ is not a square in $\GF(q)$, i.e., $q \equiv -1 \pmod 4$. 
\end{proof}

	\begin{lemma}\label{tang}
Every line of $\PG(3,q)$ not contained in $\pi$ is tangent to exactly one elliptic quadric of $\cal P$.
	\end{lemma}
	\begin{proof}
Let $\lambda\in\GF(q)$ and let $P$ be a point of the elliptic quadric $\cE_\lambda$ of $\cal P$. Let $T_P$ be the pencil of lines through $P$ in $P^{\perp_\lambda}$. In order to prove the result it is enough to show that a line $\ell$ of $T_P$ is either secant or external to a non--degenerate quadric of $\cal P$ distinct from $\cE_\lambda$. The plane $\pi$ meets $P^{\perp_\lambda}$ in a line $r$ and $\cE_{\lambda'}$, $\lambda'\ne\lambda$, in a non--degenerate conic ${\cal C}_{\lambda '}$, $\lambda'\in\GF(q)\setminus\{\lambda\}$. Then $P$, $r$, ${\cal C}_{\lambda'}$, $\lambda'\in\GF(q)\setminus\{\lambda\}$, form a pencil of quadrics of $\pi$.  From \cite[Table 7.7]{H1}, $r$ is the polar line of $P$ with respect to ${\cal C}_{\lambda'}$. Hence, $P$  is an interior point with respect to ${\cal C}_{\lambda'}$ and the result follows.
	\end{proof}

\begin{remark}\label{square}
{\rm Note that $\cE_{\lambda} \subseteq \cO_s$ if and only if $- \lambda$ is a non--zero square in $\GF(q)$.}
\end{remark}

\begin{remark}\label{oss}
{\rm Let $Q_{\lambda}$ be the quadratic form associated to $\cE_{\lambda}$ (then $Q = Q_0$). For a point of $\pi$ the evaluation of $Q_{\lambda}$ is the same for all $\lambda \in \GF(q)$.} 
\end{remark}

Let $\pi_0 = O_s \cap \pi$ and $\pi_1 = O_n \cap \pi$. Then $|\pi_0| = |\pi_1| = q(q+1)/2$. 
We need the following result.

\begin{lemma}\label{lemma3}
Let $R$ be a point of $\cE_{\lambda} \setminus \{U_3\}$, $\lambda \ne 0$, let $\ell$ be a line meeting $\cE_{\lambda}$ exactly in $R$ and let $P = \pi \cap \ell$. If $P \in \pi_0$, then 
$$
\vert \ell \cap \cE \vert=
\begin{cases} 
0 & \mbox{ if } \quad \lambda \mbox{ is a non--square in } \GF(q)\\
2 & \mbox{ if } \quad \lambda \mbox{ is a square in } \GF(q) 
\end{cases} .
$$   
If $P \in \pi_1$, then 
$$
\vert \ell \cap \cE \vert=
\begin{cases} 
2 & \mbox{ if } \quad \lambda \mbox{ is a non--square in } \GF(q)\\
0 & \mbox{ if } \quad \lambda \mbox{ is a square in } \GF(q) 
\end{cases} .
$$   
\end{lemma}
\begin{proof}
Since there exists a subgroup of $K$ of order $q^2$ which permutes in a single orbit the $q^2$ points of an elliptic quadric of $\cP$, w.l.o.g., we can choose the point $R$ as the point $(0,0,-\lambda,1) \in \cE_{\lambda}$. Then $\ell$ is contained in $R^{\perp_{\lambda}}$. Assume that $P \in \pi_0$. Straightforward calculations show that $P$ is a point having coordinates $(x,y,0,0)$, where $x^2 - \omega y^2$ is a non--zero square in $\GF(q)$ and the line $\ell$, apart from $P$, contains the points $(\mu x, \mu y, -\lambda, 1)$, $\mu \in \GF(q)$. Note that $(\mu x, \mu y, -\lambda, 1) \in \cE$ if and only if $\lambda = \mu^2 (x^2 - \omega y^2)$, that is if and only if $\lambda$ is a square in $\GF(q)$. Analogously, if $P \in \pi_1$.  
\end{proof}

Let $\lambda_1 \ne 0$ be a fixed square in $\GF(q)$ and let $\lambda_2$ be a fixed non--square in $\GF(q)$. Consider the following sets of lines:
$$
t_0 = \{r \in \cL_0 \; : \; | r \cap \pi_0 | = q\}, \; t_1 = \{r \in \cL_1 \; : \; | r \cap \pi_1 | = q\}, 
$$
$$
\cT_{10} = \{r \in \cL_2 \; : \; | r \cap \cE_{\lambda_1}| = 1, | r \cap \pi_0 | = 1\}, 
$$
$$
\cT_{11} = \{r \in \cL_3 \; : \; | r \cap \cE_{\lambda_1}| = 1, | r \cap \pi_1 | = 1\},   
$$ 
$$
\cT_{20} = \{r \in \cL_3 \; : \; | r \cap \cE_{\lambda_2}| = 1, | r \cap \pi_0 | = 1 \}, 
$$
$$
\cT_{21} = \{r \in \cL_2 \; : \; | r \cap \cE_{\lambda_2}| = 1, | r \cap \pi_1 | = 1\} .
$$ 
Then $|t_0| = |t_1| = (q+1)/2$ and $|\cT_{10}| = |\cT_{11}| = |\cT_{20}| = |\cT_{21}| = q^2(q+1)/2$. 

Let $\cA = \cT_{11} \cup \cT_{20}$ and $\cB = \cT_{10} \cup \cT_{21}$. From Lemma \ref{lemma3}, we have that $\cA$ is a set consisting of $q^2(q+1)$ external lines to $\cE$ and $\cB$ is a set consisting of $q^2(q+1)$ secant lines to $\cE$.

For a line $\ell$ of $\PG(3,q)$, we define the following line--sets: 
$$
\cA_{\ell} = \{ r \in \cA \; : \; | r \cap \ell | = 1\}, \; \cB_{\ell} = \{ r \in \cB \; : \; | r \cap \ell | = 1\}.
$$

\begin{remark}\label{oss1}
{\rm Taking into account Remark \ref{oss}, by construction, we have the following:
\begin{itemize}
\item the lines in $\cT_{11} \cup t_1$ are all the $(q+1)(q^2+1)/2$ tangent lines to $\cE_{\lambda_1}$ having $q$ non--square points with respect to $Q_{\lambda_1}$; 
\item the lines in $\cT_{10} \cup t_0$ are all the $(q+1)(q^2+1)/2$ tangent lines to $\cE_{\lambda_1}$ having $q$ square points with respect to $Q_{\lambda_1}$;
\item the lines in $\cT_{20} \cup t_0$ are all the $(q+1)(q^2+1)/2$ tangent lines to $\cE_{\lambda_2}$ having $q$ square points with respect to $Q_{\lambda_2}$;
\item the lines in $\cT_{21} \cup t_1$ are all the $(q+1)(q^2+1)/2$ tangent lines to $\cE_{\lambda_2}$ having $q$ non--square points with respect to $Q_{\lambda_2}$.  
\end{itemize}}
\end{remark}

\begin{lemma}\label{lemma4}
Let $\ell$ be a line of $\PG(3,q)$ such that $\ell \notin \cA \cup \cB$, then $|\cA_{\ell}| = |\cB_{\ell}|$.
\end{lemma}
\begin{proof}
From Lemma \ref{tang}, the line $\ell$ either is not tangent neither to $\cE_{\lambda_1}$ nor to $\cE_{\lambda_2}$, or $U_3 \in \ell \subset \pi$ and $\ell$ is tangent to both $\cE_{\lambda_1}$ and $\cE_{\lambda_2}$. Observe that if $\ell$ is not tangent neither to $\cE_{\lambda_1}$ nor to $\cE_{\lambda_2}$, then, from Remark \ref{oss1} and Lemma \ref{lemma1}, each of the following line--sets: $\cT_{11} \cup t_1$, $\cT_{20} \cup t_0$, $\cT_{10} \cup t_0$, $\cT_{21} \cup t_1$, contains $(q+1)^2/2$ lines meeting $\ell$ in a point. We consider several cases.

\bigskip
\fbox{Case 1): $|\ell \cap \cE_{\lambda_1}| = |\ell \cap \cE_{\lambda_2}| = 0$}

\begin{itemize}
\item[{\bf a)}] Assume first that $|\ell \cap \pi_0| = 1$. In this case, there is exactly one line of $t_0$ meeting $\ell$ in a point and there are no lines of $t_1$ meeting $\ell$ in a point. It follows that $|\cA_\ell| = (q+1)^2/2+(q+1)^2/2-1$. Analogously, there are $(q+1)^2/2$ lines in $\cT_{10} \cup t_0$ meeting $\ell$ and $(q+1)^2/2$ lines in $\cT_{21} \cup t_1$ meeting $\ell$. Hence, $|\cB_\ell| = (q+1)^2/2+(q+1)^2/2-1$.  
\item[{\bf b)}] If $|\ell \cap \pi_1| = 1$, then there is exactly one line of $t_1$ meeting $\ell$ in a point and there are no lines of $t_0$ meeting $\ell$ in a point. Hence, we get again $|\cA_\ell| = |\cB_{\ell}| = q(q+2)$.
\item[{\bf c)}] If $|\ell \cap \pi_0| = |\ell \cap \pi_1| = (q+1)/2$, then there are $(q+1)/2$ lines of $t_1$ meeting $\ell$ in a point and $(q+1)/2$ lines of $t_0$ meeting $\ell$ in a point. It follows that $|\cA_\ell| = |\cB_{\ell}| = (q+1)^2/2+(q+1)/2-(q+1)^2/2-(q+1)/2 = q(q+1)$.
\end{itemize}

\fbox{Case 2): $|\ell \cap \cE_{\lambda_1}| = 2, |\ell \cap \cE_{\lambda_2}| = 0$ or $|\ell \cap \cE_{\lambda_1}| = 0, |\ell \cap \cE_{\lambda_2}| = 2$ }
\bigskip

Repeating the same argument as in Case 1), {\bf a)}, or Case 1), {\bf b)}, according as $|\ell \cap \pi_0| = 1$ or $|\ell \cap \pi_1| = 1$, we obtain $|\cA_{\ell}| = |\cB_{\ell}| = q(q+2)$.

\bigskip
\fbox{Case 3): $|\ell \cap \cE_{\lambda_1}| = 2, |\ell \cap \cE_{\lambda_2}| = 2$}
\bigskip

Repeating the same argument as in Case 1), {\bf a)}, or Case 1), {\bf b)}, according as $|\ell \cap \pi_0| = 1$ or $|\ell \cap \pi_1| = 1$, we obtain $|\cA_{\ell}| = |\cB_{\ell}| = q(q+2)$. If $U_3 \in \ell$, then there are $(q+1)/2$ lines of $t_1$ meeting $\ell$ in a point and $(q+1)/2$ lines of $t_0$ meeting $\ell$ in a point. It follows that $|\cA_\ell| = |\cB_{\ell}| = q(q+1)$.

\bigskip
\fbox{Case 4): $|\ell \cap \cE_{\lambda_1}| = |\ell \cap \cE_{\lambda_2}| = 1$}
\bigskip

There are two possibilities: either $\ell \in t_0$ or $\ell \in t_1$.
\begin{itemize}
\item[{\bf a)}] $\ell \in t_0$. In this case, from Remark \ref{oss1} and Lemma \ref{lemma2}, each of the following line--sets: $\cT_{10} \cup t_0$, $\cT_{20} \cup t_0$, contains $q^2+(q-1)/2$ lines meeting $\ell$ in a point. Analogously,  each of the following line--sets: $\cT_{11} \cup t_1$, $\cT_{21} \cup t_1$, has $(q+1)/2$ elements meeting $\ell$ in a point. On the other hand, $t_0$ contains $(q-1)/2$ lines intersecting $\ell$ in a point and $t_1$ contains $(q+1)/2$ lines intersecting $\ell$ in a point. Hence, $|\cA_{\ell}| = |\cB_\ell| = q^2$.  
\item[{\bf b)}] $\ell \in t_1$. As in the previous case, we get again $|\cA_\ell| = |\cB_{\ell}| = q^2$.
\end{itemize}
The proof is now complete.
\end{proof}

\begin{lemma}\label{lemma5}
Let $\ell$ be a line of $\PG(3,q)$.
\begin{itemize} 
\item If $\ell \in \cA$, then $|\cA_{\ell}| = \frac{3q^2+3q-2}{2}$ and $|\cB_{\ell}| = \frac{q^2+3q}{2}$;
\item if $\ell \in \cB$, then $|\cA_{\ell}| = \frac{q^2+3q}{2}$ and $|\cB_{\ell}| = \frac{3q^2+3q-2}{2}$.
\end{itemize}
\end{lemma}
\begin{proof}
Assume first that $\ell \in \cA$ and in particular that $\ell \in \cT_{11}$. From Remark \ref{oss1} and Lemma \ref{lemma2}, there are $q^2+(q-1)/2$ lines of $\cT_{11} \cup t_1$ meeting $\ell$ in a point, whereas, there are $(q+1)/2$ lines of $\cT_{10}$ meeting $\ell$ in a point. Also, from Remark \ref{oss1} and Lemma \ref{lemma1}, each of the following line--sets: $\cT_{20} \cup t_0$, $\cT_{21} \cup t_1$ contains $(q+1)^2/2$ lines meeting $\ell$ in a point. Since there is exactly one line of $t_1$ meeting $\ell$ in a point and there are no lines of $t_0$ meeting $\ell$ in a point, it follows that $|\cA_{\ell}| = q^2+ (q-1)/2 - 1 + (q+1)^2/2 = (3q^2+3q-2)/2$ and $|\cB_{\ell}| = (q+1)^2/2 + (q+1)/2 -1 = (q^2+3q)/2$. Similarly, if $\ell \in \cT_{20}$.

If $\ell \in \cB$, repeating the previous arguments, we obtain the desired result.  
\end{proof}

We are ready to prove the following result.

	\begin{theorem}\label{cl}
Let $\cL$ be a Cameron--Liebler line class with parameter $(q^2+1)/2$ such that $\cA \subset \cL$ and $|\cB \cap \cL| = 0$. Then the set $\bar{\cL} = (\cL \setminus \cA) \cup \cB$ is a Cameron--Liebler line class with parameter $(q^2+1)/2$.
	\end{theorem} 
\begin{proof}
Since $\cL$ is a Cameron--Liebler line class with parameter $(q^2+1)/2$, we have that $|\{r \in \cL \; : \; |r \cap \ell| \ge 1\}|$ equals  $q^2+ (q+1)(q^2+1)/2$ if $\ell \in \cL$, or $(q + 1)(q^2+1)/2$ if $\ell \notin \cL$. 

Let $\ell$ be a line of $\PG(3,q)$. 
\begin{itemize}
\item  If $\ell \in \cL \setminus (\cA \cup \cB)$, then $\ell \in \bar{\cL}$. From Lemma \ref{lemma4}, it follows that $|\{r \in \bar{\cL} \; : \; |r \cap \ell| \ge 1\}|$ equals  $q^2+ (q+1)(q^2+1)/2$. 
\item If $\ell \notin \cL \cup \cA \cup \cB$, then $\ell \notin \bar{\cL}$. From Lemma \ref{lemma4}, it follows that $|\{r \in \bar{\cL} \; : \; |r \cap \ell| \ge 1\}|$ equals $(q + 1)(q^2+1)/2$. 
\item If $\ell \in \cA$, then $\ell \in \cL \setminus \bar{\cL}$. From Lemma \ref{lemma5}, we have that $|\{r \in \bar{\cL} \; : \; |r \cap \ell| \ge 1\}|$ equals $q^2 + (q + 1)(q^2+1)/2 - (3q^2+3q-2)/2 + (q^2+3q)/2 - 1 = (q+1)(q^2+1)/2$.
\item If $\ell \in \cB$, then $\ell \in \bar{\cL} \setminus \cL$. From Lemma \ref{lemma5}, we have that $|\{r \in \bar{\cL} \; : \; |r \cap \ell| \ge 1\}|$ equals $(q + 1)(q^2+1)/2 + (3q^2+3q-2)/2 - (q^2+3q)/2 + 1 = q^2 + (q+1)(q^2+1)/2$. 
\end{itemize}
The proof is now complete.
\end{proof}
	
We consider $\cL$ being $\cL'$ and we denote by $\cL''' = (\cL' \setminus \cA) \cup \cB$. Then, from Theorem \ref{cl}, $\cL'''$ is a Cameron--Liebler line class with parameter $(q^2+1)/2$.

In what follows, we show that $\cL'''$ is left invariant by a group of order $q^2(q+1)$. We shall find it helpful to associate to a projectivity of $\PGL(4,q)$ a matrix of $\GL(4,q)$. We shall consider the points as column vectors, with matrices acting on the left.

\begin{prop}
The Cameron--Liebler line class $\cL'''$ admits a subgroup of $K$ of order $q^2(q+1)$.
\end{prop}
\begin{proof}
Let $\Psi$ be the subgroup of $\PGL(4,q)$ whose elements are associated to the following matrices:
$$
\left( \begin{array}{cccc}
1 & 0 & 0 & -x \\
0 & 1 & 0 & -y \\
2x & -2\omega y & 1 & \omega y^2 - x^2 \\
0 & 0 & 0 & 1  
\end{array} \right) ,
$$
where $x, y \in \GF(q)$. Then $\Psi$ is an elementary abelian group of order $q^2$. Easy computations show that if $g \in \Psi$ and $P \in \cE_{\lambda}$, then $P^g \in \cE_{\lambda}$ and $\Psi$ acts transitively on $\cE_{\lambda} \setminus \{U_3\}$, $\lambda \in \GF(q)$; furthermore, the evaluation of $Q_{\lambda'}$, $\lambda' \ne \lambda$, is the same for both $P$ and $P^g$. 
Let $\Phi$ be the subgroup of $\PGL(4,q)$ whose elements are associated to the following matrices:
$$
\left( \begin{array}{cccc}
z & wt & 0 & 0 \\
t & z & 0 & 0 \\
0 & 0 & u & 0 \\
0 & 0 & 0 & u  
\end{array} \right) ,
$$
where $z, t, u \in \GF(q)$ are such that $z^2-wt^2=u^2$. Note that the previous equation holds true if and only if the line of $\PG(2,q)$ joining the points $(0,0,1)$ and $(z,t,0)$ is secant to the conic $\cD: X_1^2-wX_2^2-X_3^2 = 0$. Since $(0,0,1)$ is an interior point with respect to $\cD$, we have that up to a scalar factor there are exactly $q+1$ triple $(z,t,u)$ such that $z^2-wt^2=u^2$. Hence $|\Phi| = q+1$. Easy computations show that if $g \in \Phi$ and $P \in \cE_{\lambda}$, then $P^g \in \cE_{\lambda}$; furthermore, the evaluation of $Q_{\lambda'}$, $\lambda' \ne \lambda$, for the point $P$ is a square if and only if it is a square for $P^g$.
Since both, $\Psi$ and $\Phi$ stabilize $U_3$, we have that the direct product $\Gamma = \Psi \times \Phi$ is a group of order $q^2(q+1)$ fixing the point $U_3$ and the plane $\pi$. Hence $\Gamma$ is a subgroup of $K$. 

The group $\Gamma$ has the orbits $\pi_0$, $\pi_1$ and $\{U_3\}$ on $\pi$. It follows that the stabilizer in $\Gamma$ of a point of $\pi \setminus \{U_3\}$ has order two. Let $t$ be a line of $\PG(3,q)$ not contained in $\pi$. From Lemma \ref{tang}, $t$ is tangent to exactly an elliptic quadric $\cE_{\lambda}$ at the point $R$. The stabilizer of $t$ in $\Gamma$ has to fix $R$ and $t \cap \pi$. Hence it has order at most two. On the other hand $(R U_3)^{\perp_{\lambda}}$ does not depend on $\lambda$ and $\Gamma$ contains the involutory biaxial homology of $G$ fixing pointwise the lines $R U_3$ and $(R U_3)^{\perp_{\lambda}}$. Hence $|t^{\Gamma}|$ equals $q^2(q+1)/2$ and $\Gamma$ fixes $\cL'''$.    
\end{proof}

Let us denote by $O_s^i$ and $O_n^i$, the sets of size $q^2(q^2+1)/2$ corresponding to points of $\PG(3,q)$ such that the evaluation of the quadratic form $Q_{\lambda_i}$, $i = 1,2$, is a non--zero square or a non--square in $\GF(q)$, respectively.

\begin{prop}\label{char1}
The characters of $\cL'''$, with respect to line--sets in planes of $\PG(3,q)$ form a subset of:
$$
\left\{ q^2 + \frac{q+1}{2}, q^2 - \frac{3(q+1)}{2}, \frac{q(q-1)}{2}+3(q+1), \frac{q(q-1)}{2}+2(q+1), \right.
$$
$$
\left. \frac{q(q-1)}{2} +q+1, \frac{q(q-1)}{2}, \frac{q(q-1)}{2}-(q+1), \frac{q(q+1)}{2}-2(q+1) \right\} .
$$ 
\end{prop}
\begin{proof}
Note that if $\sigma$ is a plane distinct from $\pi$ and not containing $U_3$, then $\sigma = P^{\perp_{\lambda_i}}$, for some $P \in \cE_{\lambda_i} \setminus \{U_3\}$. In particular we may assume that $P = (0,0,-\lambda_i,1)$.

The plane $\pi$ contains $q^2+(q+1)/2$ lines of $\cL'''$. 

Let $\sigma$ be a plane distinct from $\pi$. 

If $\sigma \cap \pi \in \cL_0$, then $\sigma$ contains $q$ lines of $\cT_{20}$ and of $\cT_{10}$ and no line of $\cT_{11}$ and of $\cT_{21}$. Since $\sigma$ contains $q(q-1)/2 + q+1$ lines of $\cL'$, we have that $\sigma$ contains $q(q-1)/2 + q+1$ lines of $\cL'''$. 

If $\sigma \cap \pi \in \cL_1$, then $\sigma$ contains $q$ lines of $\cT_{11}$ and of $\cT_{21}$ and no line of $\cT_{20}$ and of $\cT_{10}$. Since $\sigma$ contains $q(q-1)/2$ lines of $\cL'$, we have that $\sigma$ contains $q(q-1)/2$ lines of $\cL'''$. 

Let $\sigma = P^\perp$, with $P \in \cE \setminus \{U_3\}$. Consider the tactical configuration whose points are the $q^2$ planes tangent to $\cE$ at some point of $\cE \setminus \{U_3\}$ and whose blocks are the $q^2(q+1)/2$ lines contained in $\cT_{11}$ (in $\cT_{20}$). It turns out that $\sigma$ contains $q+1$ lines of $\cT_{11}$ ($\cT_{20}$). Since $\sigma$ contains $q^2+(q+1)/2$ lines of $\cL'$, we have that $\sigma$ contains $q^2-3(q+1)/2$ lines of $\cL'''$. 

\smallskip
\fbox{Let $\sigma = P^{\perp_{\lambda_1}}, P \in \cE_{\lambda_1} \setminus \{U_3\}$.}
\smallskip

Assume that $q \equiv 1 \pmod 4$. Taking into account Lemma \ref{square1}, we have that $\sigma$ contains no line of $\cL_0$. 

If $\lambda_1 - \lambda_2$ is a non--square, then $(P^{\perp_{\lambda_1}})^{\perp_{\lambda_2}} \in O_n^2$. The plane $\sigma$ contains $(q+1)/2$ lines of $\cT_{11}$ and of $\cT_{10}$, no line of $\cT_{21}$ and $q+1$ lines of $\cT_{20}$. Since $\sigma$ contains $q(q-1)/2$ lines of $\cL'$, we have that $\sigma$ contains $q(q-1)/2$ lines of $\cL'''$. 

If $\lambda_1 - \lambda_2$ is a square, then $(P^{\perp_{\lambda_1}})^{\perp_{\lambda_2}} \in O_s^2$. The plane $\sigma$ contains $(q+1)/2$ lines of $\cT_{11}$ and of $\cT_{10}$, no line of $\cT_{20}$ and $q+1$ lines of $\cT_{21}$. Since $\sigma$ contains $q(q-1)/2$ lines of $\cL'$, we have that $\sigma$ contains $q(q-1)/2+q+1$ lines of $\cL'''$.

Assume that $q \equiv -1 \pmod 4$. Taking into account Lemma \ref{square1}, we have that $\sigma$ contains $q+1$ lines of $\cL_0$.

If $\lambda_1 - \lambda_2$ is a non--square, then $(P^{\perp_{\lambda_1}})^{\perp_{\lambda_2}} \in O_n^2$. The plane $\sigma$ contains $(q+1)/2$ lines of $\cT_{11}$ and of $\cT_{10}$, no line of $\cT_{21}$ and $q+1$ lines of $\cT_{20}$. Since $\sigma$ contains $q(q-1)/2$ lines of $\cL'$, we have that $\sigma$ contains $q(q-1)/2$ lines of $\cL'''$. 

If $\lambda_1 - \lambda_2$ is a square, then $(P^{\perp_{\lambda_1}})^{\perp_{\lambda_2}} \in O_s^2$. The plane $\sigma$ contains $(q+1)/2$ lines of $\cT_{11}$ and of $\cT_{10}$, no line of $\cT_{20}$ and $q+1$ lines of $\cT_{21}$. Since $\sigma$ contains $q(q-1)/2$ lines of $\cL'$, we have that $\sigma$ contains $q(q-1)/2+q+1$ lines of $\cL'''$.

\smallskip
\fbox{Let $\sigma = P^{\perp_{\lambda_2}}, P \in \cE_{\lambda_2} \setminus \{U_3\}$.}
\smallskip

Assume that $q \equiv 1 \pmod 4$. Taking into account Lemma \ref{square1}, we have that $\sigma$ contains $q+1$ lines of $\cL_0$. 

If $\lambda_2 - \lambda_1$ is a non--square, then $(P^{\perp_{\lambda_2}})^{\perp_{\lambda_1}} \in O_n^1$. The plane $\sigma$ contains $(q+1)/2$ lines of $\cT_{20}$ and of $\cT_{21}$, no line of $\cT_{11}$ and $q+1$ lines of $\cT_{10}$. Since $\sigma$ contains $q(q-1)/2+q+1$ lines of $\cL'$, we have that $\sigma$ contains $q(q-1)/2+2(q+1)$ lines of $\cL'''$. 

If $\lambda_2 - \lambda_1$ is a square, then $(P^{\perp_{\lambda_2}})^{\perp_{\lambda_1}} \in O_s^1$. The plane $\sigma$ contains $(q+1)/2$ lines of $\cT_{20}$ and of $\cT_{21}$, no line of $\cT_{10}$ and $q+1$ lines of $\cT_{11}$. Since $\sigma$ contains $q(q-1)/2+q+1$ lines of $\cL'$, we have that $\sigma$ contains $q(q-1)/2$ lines of $\cL'''$.

Assume that $q \equiv -1 \pmod 4$. Taking into account Lemma \ref{square1}, we have that $\sigma$ contains no line of $\cL_0$.

If $\lambda_2 - \lambda_1$ is a non--square, then $(P^{\perp_{\lambda_2}})^{\perp_{\lambda_1}} \in O_n^1$. The plane $\sigma$ contains $(q+1)/2$ lines of $\cT_{20}$ and of $\cT_{21}$, no line of $\cT_{10}$ and $q+1$ lines of $\cT_{11}$. Since $\sigma$ contains $q(q-1)/2$ lines of $\cL'$, we have that $\sigma$ contains $q(q-1)/2-(q+1)$ lines of $\cL'''$. 

If $\lambda_2 - \lambda_1$ is a square, then $(P^{\perp_{\lambda_2}})^{\perp_{\lambda_1}} \in O_s^1$. The plane $\sigma$ contains $(q+1)/2$ lines of $\cT_{20}$ and of $\cT_{21}$, no line of $\cT_{11}$ and $q+1$ lines of $\cT_{10}$. Since $\sigma$ contains $q(q-1)/2$ lines of $\cL'$, we have that $\sigma$ contains $q(q-1)/2+q+1$ lines of $\cL'''$.

\smallskip
\fbox{Let $\sigma = P^{\perp_{\lambda_3}}, P \in \cE_{\lambda_3} \setminus \{U_3\}$, $\lambda_3 \in \GF(q) \setminus \{0, \lambda_1, \lambda_2\}$.}
\smallskip

Taking into account Lemma \ref{square1}, we have that $(P^{\perp_{\lambda_3}})^{\perp} \in O_s$ if and only if $\lambda_3$ is a square in $\GF(q)$. 

Assume that $\lambda_3$ is a square in $\GF(q)$. The following possibilities arise: 

\begin{itemize}
\item $\lambda_3 - \lambda_1$, $\lambda_3-\lambda_2$ are squares in $\GF(q)$ and then $(P^{\perp_{\lambda_3}})^{\perp_{\lambda_1}} \in O_s^1$ and $(P^{\perp_{\lambda_3}})^{\perp_{\lambda_2}} \in O_s^2$. If $q \equiv -1 \pmod 4$, $\sigma$ contains $q+1$ lines of $\cT_{20}$ and of $\cT_{10}$ and no line of $\cT_{21}$ and of $\cT_{11}$. Since $\sigma$ contains $q(q-1)/2 + q+1$ lines of $\cL'$, we have that $\sigma$ contains $q(q-1)/2+q+1$ lines of $\cL'''$.  If $q \equiv 1 \pmod 4$, $\sigma$ contains $q+1$ lines of $\cT_{11}$ and of $\cT_{21}$ and no line of $\cT_{20}$ and of $\cT_{10}$. Since $\sigma$ contains $q(q-1)/2$ lines of $\cL'$, we have that $\sigma$ contains $q(q-1)/2$ lines of $\cL'''$; 

\item $\lambda_3 - \lambda_1$ is a non--square and $\lambda_3-\lambda_2$ is a square in $\GF(q)$, then $(P^{\perp_{\lambda_3}})^{\perp_{\lambda_1}} \in O_n^1$ and $(P^{\perp_{\lambda_3}})^{\perp_{\lambda_2}} \in O_s^2$. If $q \equiv -1 \pmod 4$, $\sigma$ contains $q+1$ lines of $\cT_{20}$ and of $\cT_{11}$ and no line of $\cT_{21}$ and of $\cT_{10}$. Since $\sigma$ contains $q(q-1)/2 + q+1$ lines of $\cL'$, we have that $\sigma$ contains $q(q-1)/2-(q+1)$ lines of $\cL'''$.  If $q \equiv 1 \pmod 4$, $\sigma$ contains $q+1$ lines of $\cT_{21}$ and of $\cT_{10}$ and no line of $\cT_{20}$ and of $\cT_{11}$. Since $\sigma$ contains $q(q-1)/2$ lines of $\cL'$, we have that $\sigma$ contains $q(q-1)/2+2(q+1)$ lines of $\cL'''$; 

\item $\lambda_3 - \lambda_1$ is a square and $\lambda_3-\lambda_2$ is a non--square in $\GF(q)$, then $(P^{\perp_{\lambda_3}})^{\perp_{\lambda_1}} \in O_s^1$ and $(P^{\perp_{\lambda_3}})^{\perp_{\lambda_2}} \in O_n^2$. If $q \equiv -1 \pmod 4$, $\sigma$ contains $q+1$ lines of $\cT_{21}$ and of $\cT_{10}$ and no line of $\cT_{20}$ and of $\cT_{11}$. Since $\sigma$ contains $q(q-1)/2 + q+1$ lines of $\cL'$, we have that $\sigma$ contains $q(q-1)/2+3(q+1)$ lines of $\cL'''$.  If $q \equiv 1 \pmod 4$, $\sigma$ contains $q+1$ lines of $\cT_{20}$ and of $\cT_{11}$ and no line of $\cT_{21}$ and of $\cT_{10}$. Since $\sigma$ contains $q(q-1)/2$ lines of $\cL'$, we have that $\sigma$ contains $q(q-1)/2-2(q+1)$ lines of $\cL'''$; 

\item $\lambda_3 - \lambda_1$, $\lambda_3-\lambda_2$ are non--squares in $\GF(q)$ and then $(P^{\perp_{\lambda_3}})^{\perp_{\lambda_1}} \in O_n^1$ and $(P^{\perp_{\lambda_3}})^{\perp_{\lambda_2}} \in O_n^2$. If $q \equiv -1 \pmod 4$, $\sigma$ contains $q+1$ lines of $\cT_{11}$ and of $\cT_{21}$ and no line of $\cT_{20}$ and of $\cT_{10}$. Since $\sigma$ contains $q(q-1)/2 + q+1$ lines of $\cL'$, we have that $\sigma$ contains $q(q-1)/2+q+1$ lines of $\cL'''$.  If $q \equiv 1 \pmod 4$, $\sigma$ contains $q+1$ lines of $\cT_{20}$ and of $\cT_{10}$ and no line of $\cT_{21}$ and of $\cT_{11}$. Since $\sigma$ contains $q(q-1)/2$ lines of $\cL'$, we have that $\sigma$ contains $q(q-1)/2$ lines of $\cL'''$; 
\end{itemize}

Assume that $\lambda_3$ is a non--square in $\GF(q)$. Arguing as above, we have the following possibilities: 

\begin{itemize}
\item $\lambda_3 - \lambda_1$, $\lambda_3-\lambda_2$ are squares in $\GF(q)$. In this case $\sigma$ contains $q(q-1)/2$ or $q(q-1)/2+q+1$ lines of $\cL'''$, according as $q \equiv -1 \pmod 4$ or $q \equiv 1 \pmod 4$;

\item $\lambda_3 - \lambda_1$ is a non--square and $\lambda_3-\lambda_2$ is a square in $\GF(q)$. In this case $\sigma$ contains $q(q-1)/2-2(q+1)$ or $q(q-1)/2+3(q+1)$ lines of $\cL'''$, according as $q \equiv -1 \pmod 4$ or $q \equiv 1 \pmod 4$;
 
\item $\lambda_3 - \lambda_1$ is a square and $\lambda_3-\lambda_2$ is a non--square in $\GF(q)$. In this case $\sigma$ contains $q(q-1)/2+2(q+1)$ or $q(q-1)/2-(q+1)$ lines of $\cL'''$, according as $q \equiv -1 \pmod 4$ or $q \equiv 1 \pmod 4$;

\item $\lambda_3 - \lambda_1$, $\lambda_3-\lambda_2$ are non--squares in $\GF(q)$. In this case $\sigma$ contains $q(q-1)/2$ or $q(q-1)/2+q+1$ lines of $\cL'''$, according as $q \equiv -1 \pmod 4$ or $q \equiv 1 \pmod 4$.
\end{itemize}
The proof is now complete.
\end{proof}

\begin{prop}\label{char2}
The characters of $\cL'''$, with respect to line--stars of $\PG(3,q)$ form a subset of:
$$
\left\{ \frac{q+1}{2}, \frac{5(q+1)}{2}, \frac{q(q+1)}{2}-2(q+1), \frac{q(q+1)}{2}-(q+1), \right.
$$
$$
\left. \frac{q(q+1)}{2}, \frac{q(q+1)}{2}+q+1, \frac{q(q+1)}{2}+2(q+1), \frac{q(q+1)}{2}+3(q+1) \right\}.
$$ 
\end{prop}
\begin{proof}
Through the point $U_3$, there pass $(q+1)/2$ lines of $\cL'''$. 

Let $P \in \cE \setminus \{U_3\}$. Consider the tactical configuration whose points are the $q^2$ points of $\cE \setminus \{U_3\}$ and whose blocks are the $q^2(q+1)/2$ lines contained in $\cT_{10}$ (in $\cT_{21}$). It turns out that through $P$ there pass $q+1$ lines of $\cT_{10}$ ($\cT_{21}$). Since $P$ is on $q+1$ lines of $\cL'$, we have that $P$ is on $5(q+1)/2$ lines of $\cL'''$. 

Let $P \in \pi_0$. Through $P$ there pass $q$ lines of $\cT_{10}$ and of $\cT_{20}$ and no line of $\cT_{11}$ and of $\cT_{21}$. Since $P$ is on $q(q+1)/2 + q+1$ lines of $\cL'$, we have that $P$ is on $q(q+1)/2 + q+1$ lines of $\cL'''$. 

Let $P \in \pi_1$. Through $P$ there pass $q$ lines of $\cT_{11}$ and of $\cT_{21}$ and no line of $\cT_{10}$ and of $\cT_{20}$. Since $P$ is on $q(q+1)/2$ lines of $\cL'$, we have that $P$ is on $q(q+1)/2$ lines of $\cL'''$. 

\smallskip
\fbox{Let $P \in \cE_{\lambda_1} \setminus \{U_3\}$.}
\smallskip

Assume that $q \equiv 1 \pmod 4$. Taking into account Remark \ref{square}, we have that $\cE_{\lambda_1} \subseteq O_s$. 

If $\lambda_2 - \lambda_1$ is a square, then $\cE_{\lambda_1} \subseteq O_s^2$. Through $P$ there pass $(q+1)/2$ lines of $\cT_{11}$ and of $\cT_{10}$, no line of $\cT_{21}$ and $q+1$ lines of $\cT_{20}$. Since $P$ is on $q(q+1)/2 + q+1$ lines of $\cL'$, we have that $P$ is on $q(q+1)/2$ lines of $\cL'''$. 

If $\lambda_2 - \lambda_1$ is a non--square, then $\cE_{\lambda_1} \subseteq O_n^2$. Through $P$ there pass $(q+1)/2$ lines of $\cT_{11}$ and of $\cT_{10}$, no line of $\cT_{20}$ and $q+1$ lines of $\cT_{21}$. Since $P$ is on $q(q+1)/2 + q+1$ lines of $\cL'$, we have that $P$ is on $q(q+1)/2 + 2(q+1)$ lines of $\cL'''$.

Assume that $q \equiv -1 \pmod 4$. Taking into account Remark \ref{square}, we have that $\cE_{\lambda_1} \subseteq O_n$. 

If $\lambda_2 - \lambda_1$ is a square, then $\cE_{\lambda_1} \subseteq O_s^2$. Through $P$ there pass $(q+1)/2$ lines of $\cT_{11}$ and of $\cT_{10}$, no line of $\cT_{21}$ and $q+1$ lines of $\cT_{20}$. Since $P$ is on $q(q+1)/2$ lines of $\cL'$, we have that $P$ is on $q(q+1)/2-(q+1)$ lines of $\cL'''$. 

If $\lambda_2 - \lambda_1$ is a non--square, then $\cE_{\lambda_1} \subseteq O_n^2$. Through $P$ there pass $(q+1)/2$ lines of $\cT_{11}$ and of $\cT_{10}$, no line of $\cT_{20}$ and $q+1$ lines of $\cT_{21}$. Since $P$ is on $q(q+1)/2$ lines of $\cL'$, we have that $P$ is on $q(q+1)/2 + q+1$ lines of $\cL'''$.

\smallskip
\fbox{Let $P \in \cE_{\lambda_2} \setminus \{U_3\}$.}
\smallskip

Assume that $q \equiv 1 \pmod 4$. Taking into account Remark \ref{square}, we have that $\cE_{\lambda_2} \subseteq O_n$. 

If $\lambda_1 - \lambda_2$ is a square, then $\cE_{\lambda_2} \subseteq O_s^1$. Through $P$ there pass $(q+1)/2$ lines of $\cT_{20}$ and of $\cT_{21}$, no line of $\cT_{11}$ and $q+1$ lines of $\cT_{10}$. Since $P$ is on $q(q+1)/2$ lines of $\cL'$, we have that $P$ is on $q(q+1)/2+q+1$ lines of $\cL'''$. 

If $\lambda_1 - \lambda_2$ is a non--square, then $\cE_{\lambda_2} \subseteq O_n^1$. Through $P$ there pass $(q+1)/2$ lines of $\cT_{20}$ and of $\cT_{21}$, no line of $\cT_{10}$ and $q+1$ lines of $\cT_{11}$. Since $P$ is on $q(q+1)/2$ lines of $\cL'$, we have that $P$ is on $q(q+1)/2-(q+1)$ lines of $\cL'''$.

Assume that $q \equiv -1 \pmod 4$. Taking into account Remark \ref{square}, we have that $\cE_{\lambda_2} \subseteq O_s$. 

If $\lambda_1 - \lambda_2$ is a square, then $\cE_{\lambda_2} \subseteq O_s^1$. Through $P$ there pass $(q+1)/2$ lines of $\cT_{20}$ and of $\cT_{21}$, no line of $\cT_{11}$ and $q+1$ lines of $\cT_{10}$. Since $P$ is on $q(q+1)/2+q+1$ lines of $\cL'$, we have that $P$ is on $q(q+1)/2+2(q+1)$ lines of $\cL'''$. 

If $\lambda_1 - \lambda_2$ is a non--square, then $\cE_{\lambda_2} \subseteq O_n^1$. Through $P$ there pass $(q+1)/2$ lines of $\cT_{20}$ and of $\cT_{21}$, no line of $\cT_{10}$ and $q+1$ lines of $\cT_{11}$. Since $P$ is on $q(q+1)/2+q+1$ lines of $\cL'$, we have that $P$ is on $q(q+1)/2$ lines of $\cL'''$.

\smallskip
\fbox{Let $P \in \cE_{\lambda_3} \setminus \{U_3\}$, $\lambda_3 \in \GF(q) \setminus \{0, \lambda_1, \lambda_2\}$.}
\smallskip

Taking into account Remark \ref{square}, we have that $P \in O_s$ if and only if $-\lambda_3$ is a square in $\GF(q)$. 

Assume that $-\lambda_3$ is a square in $\GF(q)$. The following possibilities arise: 

\begin{itemize}
\item $\lambda_1 - \lambda_3$, $\lambda_2-\lambda_3$ are squares in $\GF(q)$ and then $P \in O_s^1 \cap O_s^2$. Through $P$ there pass $q+1$ lines of $\cT_{20}$ and of $\cT_{10}$ and no line of $\cT_{21}$ and of $\cT_{11}$. Since $P$ is on $q(q+1)/2 +q+1$ lines of $\cL'$, we have that $P$ is on $q(q+1)/2+q+1$ lines of $\cL'''$;
\item $\lambda_1 - \lambda_3$ is a square and $\lambda_2-\lambda_3$ is a non--square in $\GF(q)$ and then $P \in O_s^1 \cap O_n^2$. Through $P$ there pass $q+1$ lines of $\cT_{21}$ and of $\cT_{10}$ and no line of $\cT_{20}$ and of $\cT_{11}$. Since $P$ is on $q(q+1)/2 + q+1$ lines of $\cL'$, we have that $P$ is on $q(q+1)/2+3(q+1)$ lines of $\cL'''$; 
\item $\lambda_1 - \lambda_3$ is a non--square and $\lambda_2-\lambda_3$ is a square in $\GF(q)$ and then $P \in O_n^1 \cap O_s^2$. Through $P$ there pass $q+1$ lines of $\cT_{11}$ and of $\cT_{20}$ and no line of $\cT_{21}$ and of $\cT_{10}$. Since $P$ is on $q(q+1)/2 + q+1$ lines of $\cL'$, we have that $P$ is on $q(q+1)/2-(q+1)$ lines of $\cL'''$; 
\item $\lambda_1 - \lambda_3$, $\lambda_2-\lambda_3$ are non--squares in $\GF(q)$ and then $P \in O_n^1 \cap O_n^2$. Through $P$ there pass $q+1$ lines of $\cT_{11}$ and of $\cT_{21}$ and no line of $\cT_{20}$ and of $\cT_{10}$. Since $P$ is on $q(q+1)/2 + q+1$ lines of $\cL'$, we have that $P$ is on $q(q+1)/2+q+1$ lines of $\cL'''$. 
\end{itemize}

Assume that $-\lambda_3$ is a non--square in $\GF(q)$. Arguing as above, we have the following possibilities: 

\begin{itemize}
\item $\lambda_1 - \lambda_3$, $\lambda_2-\lambda_3$ are squares in $\GF(q)$ and then $P \in O_s^1 \cap O_s^2$. In this case $P$ is on $q(q+1)/2$ lines of $\cL'''$;
\item $\lambda_1 - \lambda_3$ is a square and $\lambda_2-\lambda_3$ is a non--square in $\GF(q)$ and then $P \in O_s^1 \cap O_n^2$. In this case $P$ is on $q(q+1)/2+2(q+1)$ lines of $\cL'''$; 
\item $\lambda_1 - \lambda_3$ is a non--square and $\lambda_2-\lambda_3$ is a square in $\GF(q)$ and then $P \in O_n^1 \cap O_s^2$. In this case $P$ is on $q(q+1)/2-2(q+1)$ lines of $\cL'''$; 
\item $\lambda_1 - \lambda_3$, $\lambda_2-\lambda_3$ are non--squares in $\GF(q)$ and then $P \in O_n^1 \cap O_n^2$. In this case $P$ is on $q(q+1)/2$ lines of $\cL'''$. 
\end{itemize}
The proof is now complete.
\end{proof}

\begin{theorem}
If $q \ge 7$ odd, the Cameron--Liebler line class $\cL'''$ is not equivalent to one of the previously known examples.
\end{theorem}
\begin{proof}
From the proof of Proposition \ref{char2}, if $P \in \cE \setminus \{U_3\}$, through $P$ there pass $5(q+1)/2$ lines of $\cL'''$. Since, for $q \ge 7$, the value $5(q+1)/2$ does not appear among the characters of $\cL'$ and $\cL''$, we may conclude that $\cL'''$ is distinct from $\cL'$ and $\cL''$. On the other hand, both examples $\cX \cup \cY$ and $\cX \cup {\cal Z}$ admit $q^2+q+1$ as a character, but from Proposition \ref{char1} and Proposition \ref{char2}, such a value does not appear as a character of $\cL'''$.   
\end{proof}

\begin{remark}\label{oss2}
{\rm Let $\Box_q$ denote the non--zero square elements of $\GF(q)$. From the proof of Proposition \ref{char1} and of Proposition \ref{char2}, if there exist $a_1, a_2, a_3, a_4, b_1, b_2, b_3, b_4 \in \GF(q) \setminus \{0, \lambda_1, \lambda_2\}$ such that $a_1, a_2, a_3, a_4 \in \Box_q$, $b_1, b_2, b_3, b_4 \notin \Box_q$ and $a_1 - \lambda_1 \in \Box_q$, $a_1 - \lambda_2 \in \Box_q$, $a_2 - \lambda_1 \in \Box_q$, $a_2 - \lambda_2 \notin \Box_q$, $a_3 - \lambda_1 \notin \Box_q$, $a_3 - \lambda_2 \in \Box_q$, $a_4 - \lambda_1 \notin \Box_q$, $a_4 - \lambda_2 \notin \Box_q$, $b_1 - \lambda_1 \in \Box_q$, $b_1 - \lambda_2 \in \Box_q$, $b_2 - \lambda_1 \in \Box_q$, $b_2 - \lambda_2 \notin \Box_q$, $b_3 - \lambda_1 \notin \Box_q$, $b_3 - \lambda_2 \in \Box_q$, $b_4 - \lambda_1 \notin \Box_q$, $b_4 - \lambda_2 \notin \Box_q$, then $\cL'''$ has exactly eight characters with respect to line--sets in planes of $\PG(3,q)$ (line--stars of $\PG(3,q)$).}
\end{remark}

\begin{remark}
{\rm Note that both, $\cL'$ and $\cL''$, are Cameron--Liebler line classes satisfying the requirements of Theorem \ref{cl} and that, starting from $\cL'$ or $\cL''$, the replacement technique described in Theorem \ref{cl} can be iterated $(q-1)/2$ times.}
\end{remark}

\begin{remark}
{\rm Computations performed with Magma \cite{Magma} suggest that starting from $\cL'$ and applying Theorem \ref{cl} (multiple derivation is allowed), apart from Bruen--Drudge's example and the example described in \cite{CP} and \cite{GMP}, we get what follows. The notation $\alpha^i$ in the character strings below stands for: there are $i$ planes containing $\alpha$ lines of the Cameron--Liebler line class or there are $i$ line--stars containing $\alpha$ lines of the Cameron--Liebler line class. 

\smallskip
\fbox{$q=7$}
A new example arises having the characters: 

\begin{itemize}
\item[$i)$] $13^{49}, 21^{126}, 29^{77}, 37^{98}, 45^{49}, 53$ with respect to line--sets in planes of $\PG(3,7)$ and $4, 12^{49}, 20^{98}, 28^{77}, 36^{126}, 44^{49}$ with respect to line--stars of $\PG(3,7)$. 
\end{itemize}

\smallskip
\fbox{$q=9$}
Three new example arise having the characters: 

\begin{itemize}
\item[$i)$] $16^{81}, 36^{207}, 46^{288}, 56^{81}, 66^{162}, 86$ with respect to line--sets in planes of $\PG(3,9)$ and $5, 25^{162}, 35^{81}, 45^{288}, 55^{207}, 75^{81}$ with respect to line--stars of $\PG(3,9)$;
\item[$ii)$] $26^{162}, 36^{207}, 46^{126}, 56^{162}, 66^{162}, 86$ with respect to line--sets in planes of $\PG(3,9)$ and $5, 25^{162}, 35^{162}, 45^{126}, 55^{207}, 65^{162}$ with respect to line--stars of $\PG(3,9)$;  
\item[$iii)$] $26^{162}, 36^{126}, 46^{288}, 56^{162}, 76^{81}, 86$ with respect to line--sets in planes of $\PG(3,9)$ and $5, 15^{81}, 35^{162}, 45^{288}, 55^{126}, 65^{162}$ with respect to line--stars of $\PG(3,9)$.
\end{itemize}

\smallskip
\fbox{$q=11$}
Five new example arise having the characters: 

\begin{itemize}
\item[$i)$] $31^{121}, 43^{121}, 55^{308}, 67^{429}, 79^{242}, 91^{121}, 103^{121}, 127$ with respect to line--sets in planes of $\PG(3,11)$ and $6, 30^{121}, 42^{121}, 54^{242}, 66^{429}, 78^{308}, 90^{121}, 102^{121}$ with respect to line--stars of $\PG(3,11)$, see Remark \ref{oss2};
\item[$ii)$] $43^{242}, 55^{550}, 67^{187}, 79^{121}, 91^{242}, 103^{121}, 127$ with respect to line--sets in planes of $\PG(3,11)$ and $6, 30^{121}, 42^{242}, 54^{121}, 66^{187}, 78^{550}, 90^{242}$ with respect to line--stars of $\PG(3,11)$;
\item[$iii)$] $43^{242}, 55^{429}, 67^{308}, 79^{363}, 115^{121}, 127$ with respect to line--sets in planes of $\PG(3,11)$ and $6, 18^{121}, 54^{363}, 66^{308}, 78^{429}, 90^{242}$ with respect to line--stars of $\PG(3,11)$;
\item[$iv)$] $31^{121}, 43^{242}, 55^{187}, 67^{187}, 79^{484}, 91^{242}, 127$ with respect to line--sets in planes of $\PG(3,11)$ and $6, 42^{242}, 54^{484}, 66^{187}, 78^{187}, 90^{242},102^{121}$ with respect to line--stars of $\PG(3,11)$;
\item[$v)$] $19^{121}, 55^{429}, 67^{308}, 79^{363}, 91^{242}, 127$ with respect to line--sets in planes of $\PG(3,11)$ and $6, 42^{242}, 54^{363}, 66^{308}, 78^{429}, 114^{121}$ with respect to line--stars of $\PG(3,11)$;

\end{itemize}

Interestingly, if we start from $\cL''$, we get the complements of the abovementioned examples. In general, it seems a difficult task to determine how many inequivalent examples of Cameron--Liebler line classes arise from Theorem \ref{cl} and we leave it as an open problem.}
\end{remark}

\par\noindent A. Cossidente\\ Dipartimento di Matematica Informatica ed Economia\\
Universit\`a della Basilicata\\ Contrada Macchia Romana\\ 85100
Potenza (ITALY),\\antonio.cossidente@unibas.it
\newline
\\
\noindent F. Pavese\\ Dipartimento di Meccanica, Matematica e Management \\
Politecnico di Bari\\ Via Orabona, 4\\ 70125
Bari (ITALY),\\francesco.pavese@poliba.it


\begin{thebibliography}{99}

\bibitem{BKLP} J. Bamberg, S. Kelly, M. Law, T. Penttila, Tight sets and $m$--ovoids of finite polar spaces, {\em J. Combin. Theory Ser. A} {\bf 114} (2007), no. 7, 1293--1314.

\bibitem{Magma} W. Bosma, J. Cannon, C. Playoust, The Magma algebra system. I. The user language. Computational algebra and number theory, {\em J. Symbolic Comput.} {\bf 24} (1997), no. 3-4, 235--265.

\bibitem{BD} A.A. Bruen, K. Drudge, The construction of Cameron--Liebler line classes in $\PG(3,q)$, {\em Finite Fields Appl.} {\bf 5} (1999), no. 1, 35--45.

\bibitem{CL} P.J. Cameron, R.A. Liebler, Tactical decompositions and orbits of projective groups, {\em Linear Algebra Appl.} {\bf 46} (1982), 91--102.

\bibitem{CP}  A. Cossidente, F. Pavese, Intriguing sets of quadrics in $\PG(5,q)$, {\em Adv. Geom.} (to appear).

\bibitem{DDMR} J. De Beule, J. Demeyer, K. Metsch, M. Rodgers, A new family of tight sets of ${\cal Q}^+(5,q)$, {\em Des. Codes Cryptogr.} {\bf 78} (2016), 655--678.

\bibitem{FMX} T. Feng, K. Momihara, Q. Xiang, Cameron--Liebler line classes with parameter $x = (q^2-1)/2$, {\em J. Comb. Theory Ser. A} {\bf 133} (2015), 307--338.

\bibitem{GMP} A.L. Gavrilyuk, I. Matkin, T. Penttila, Derivation of Cameron--Liebler line classes, {\em Des. Codes. Cryptogr.} DOI 10.1007/s10623-017-0338-4.

\bibitem{GM} A.L. Gavrilyuk, K. Metsch, A modular equality for Cameron--Liebler line classes, {\em J. Combin. Theory Ser. A} {\bf 127} (2014), 224--242. 

\bibitem{GP} P. Govaerts, T. Penttila, Cameron--Liebler line classes in $\PG(3,4)$, {\em Bull. Belg. Math. Soc. Simon Stevin} {\bf 12} (2005), no. 5, 793--804.

%\bibitem{H} J.W.P. Hirschfeld, {\em Finite Projective Spaces of Three Dimensions}, Oxford University Press, %Oxford, 1991.

\bibitem{H1} J.W.P. Hirschfeld, {\em Projective Geometries over Finite Fields}, Oxford University Press, %Oxford, 1998.

%\bibitem{magma} J. Cannon, C. Playoust, {\em An introduction to MAGMA}, University of Sidney,
%Sidney, Australia, 1993.

\bibitem{M} K. Metsch, An improved bound on the existence of Cameron--Liebler line classes, {\em J. Comb. Theory Ser. A} {\bf 121} (2014), 89--93.

\bibitem{P} T. Penttila, Cameron--Liebler line classes in $\PG(3,q)$, {\em Geom. Ded.} {\bf 37} (1991), 245--252.

\bibitem{MR} M. Rodgers, {\em On some new examples of Cameron--Liebler line classes}, Ph.D. Thesis, University of Colorado, Denver, 2012.

\end{thebibliography}
    \end{document}